\documentclass[a4paper,10pt]{amsart}
\usepackage{bbm}
\usepackage{a4wide}

\usepackage{float}
\usepackage{euscript,eufrak,verbatim}
\usepackage{graphicx}
\usepackage[usenames]{color}
\usepackage[colorlinks,linkcolor=red,anchorcolor=blue,citecolor=blue]{hyperref}
\usepackage{amsmath}
\usepackage{amsthm}
\usepackage[all]{xy}
\usepackage{graphicx} 
\usepackage{amssymb}

\newtheorem{thm}{Theorem}[section]
\newtheorem{prop}[thm]{Proposition}

\newtheorem{lem}[thm]{Lemma}
\newtheorem{cor}[thm]{Corollary}

\newtheorem{ques}[thm]{Question}

\newtheorem{rem}[thm]{Remark}

\def\N{\mathbb{N}}

\def\N{\mathbb{N}}
\def\R{\mathbb{R}}

\def\LL{\mathcal{L}}
\def\MM{\mathcal{M}}

\def\Id{\text{\rm Id}}

\def\diam{\text{\rm diam}}

\makeatletter 
\@addtoreset{equation}{section}
\numberwithin{equation}{section}

\title[Zero-dimensional and symbolic extensions of topological flows]{Zero-dimensional and symbolic extensions of topological flows}

\author{David Burguet
}
\address
{Sorbonne Universite, LPSM, 75005 Paris, France}
\email{david.burguet@upmc.fr}

\author{Ruxi Shi
}
\address
{Institute of Mathematics, Polish Academy of Sciences, ul. \'Sniadeckich 8, 00-656 Warszawa, Poland}
\email{rshi@impan.pl}
\begin{document}
	
	\maketitle

\begin{abstract}
A zero-dimensional  (resp. symbolic) flow is a suspension flow over a zero-dimensional system (resp. a subshift). We show that any topological flow admits a principal extension by a zero-dimensional flow. Following  \cite{burguet2019symbolic} we deduce that  any topological flow admits an extension by 
a symbolic flow if and only if  its time-$t$ map admits an extension by a subshift for any $t\neq 0$.  Moreover the existence of such an extension is preserved  under orbit equivalence for regular topological flows, but this property does  not hold more true for singular flows. Finally we investigate symbolic extensions for singular suspension  flows. In particular, the suspension flow over the full shift on $\{0,1\}^{\mathbb Z}$ with a  roof function  $f$  vanishing at the zero sequence $0^\infty$ admits a principal symbolic extension or not depending on the smoothness of $f$ at $0^\infty$.
\end{abstract}

\section{Introduction}

A symbolic extension of a discrete topological system $(X,T)$ is a topological extension $\pi:(Y,S)\rightarrow (X,T)$ by a subshift $(Y,S)$. 
Existence of symbolic extensions and their entropy are related with weak expansive entropy properties of the system \cite{BDz2004}.  Building on \cite{BDz2004} the first author and T.Downarowicz also investigate uniform generators, which are symbolic extensions with a Borel  embedding $\psi:(X,T)\rightarrow (Y,S)$ with $\pi\circ \psi=\Id_X$.  

To study  symbolic extensions for discrete systems, M. Boyle and T. Downarowicz \cite{BDz2004} have developed a new entropy theory. The first step in their construction of symbolic extensions consists in building a zero-dimensional principal extension (see Section \ref{exten} for precise definitions). In \cite{BDz2004} this is done by using the small boundary property for  finite topological entropy systems with a minimal factor \cite{L99}. In this way one even get a strongly isomorphic zero-dimensional extension. Later T. Downarowicz and D. Huczek gave a constructive proof of a zero-dimensional principal extension for any topological system (with finite topological entropy or not).

More recently the first author  developed in  \cite{burguet2019symbolic}  a theory of uniform generators for topological  regular (i.e. without fixed points) flows. In this context a symbolic extension (resp. uniform generator) is a topological extension by a regular suspension flow over a subshift (resp.  with an embedding). To investigate existence of uniform generators he considered topological flows satisfying the so-called small flow boundary property\footnote{The small flow boundary property (for topological flows) is an analogy to the small boundary property (for discrete topological systems). This notion is introduced by the first author in \cite{burguet2019symbolic}.}. Under this extra assumption, the flow admits a  strongly isomorphic zero-dimensional extension. Here we  investigate only symbolic extensions for topological flows and not uniform generators for these, so that to be reduced to the case of   zero-dimensional suspension flows, we only need to build a  principal zero-dimensional extension (this is achieved in Section \ref{zeroext}). In this way we may also  consider  the symbolic extensions of a topological singular flow \footnote{An isomorphic extension of a singular flow is necessarily singular (so that if one wants to develop a theory of uniform generators one should consider singular suspension flows over a subshift).}.
In Section \ref{appli} we build on \cite{burguet2019symbolic} to study symbolic extension for general topological  flows. In particular we show that two regular topological flows which are orbit equivalent either  both admit a (resp. principal) symbolic extension or not. This result was only proved in \cite{burguet2019symbolic} for regular flows with the small flow boundary property. We also give counterexamples for singular flows. 

In the rest part (Section \ref{sec:Singular suspension flows}) we define singular suspension flows as suspension flows over a discrete system $(X,T)$ with a roof function vanishing only at a fixed point of $T$. We investigate their symbolic extensions. In the case $(X,T)$ is expansive, an  associated singular suspension flow with finite topological entropy does not admit in general a principal  symbolic extension. This depends on the behaviour of the roof function near the singularity. For $(X,T)$ being the two full shift  we illustrate this phenomenon by computing  the symbolic extension entropy function for a large class of roof functions. In fact we explicitly build a symbolic extension of the corresponding  singular suspension flow.

\section{Zero-dimensional principal extension}\label{zeroext}

By a  (discrete) topological system $(X,T)$, we mean that $X$ is a compact metrizable space and $T: X\to X$ is continuous. Moreover, the system is said to be {\it invertible} if $T$ is a homeomorphism. A pair $(X, \Phi)$ is called a {\it topological semi-flow} (resp. {\it topological flow}) if $X$ is a metrizable compact space and $\Phi=(\phi_t)_{t\ge 0}$ (resp. $\Phi=(\phi_t)_{t\in \mathbb R}$) are continuous maps from $X$ to itself satisfying that $\phi_0(x)=x$ and $\phi_t(\phi_s(x))=\phi_{t+s}(x)$ for all $t,s\geq 0$ (resp. $t,s\in \mathbb R$). A point $x\in X$ is a fixed point of $\Phi$ when $\phi_t(x)=x$ for all $t\geq0$. A flow without fixed point is called {\it regular}, otherwise the flow is said {\it singular}. We let $\mathcal M_{T}(X)$ or $\mathcal M_{\Phi}(X)$ (resp. $\mathcal M^e_{T}(X)$ or $\mathcal M^e_{\Phi}(X)$) be the compact set of  Borel probability measures invariant (resp. ergodic)  by the topological system or flow. We recall that the measure-theoretic entropy $h^{\Phi}(\mu)$ of  $\mu\in \mathcal M_{\Phi}(X)$ is defined as the the entropy of its time-$1$ map, i.e. $h_{\Phi}(\mu)=h_{\phi_1}(\mu)$. Through this paper, we use the following notations:  $\R_{\ge 0}=[0, \infty)$ and $\R_{> 0}=(0, \infty)$.

\subsection{Suspension semi-flows}\label{sec:Suspension}

Let $(X, d)$ be a compact metric space and $T:X\to X$ a continuous map. Let $f: X\to \mathbb{R}_{>0}$ be a continuous map. 
Let $\overline{X^f}$ be the compact subset of $X\times \mathbb R_{\geq 0}$ defined by 
\begin{align}\label{suspension}
\overline{X^{f}}:=\{(x,t): 0\le t\le f(x), x\in X \}.
\end{align}

We consider the equivalence relation  $\sim$ on $\overline{X^f}$   with $ (x,f(x))\sim (Tx,0)$ for all $x\in X$ and we denote by $X^f$ the quotient space $\overline{X^f}/\sim$ endowed with the quotient topology. By abuse of notations  we also write $(x,t)$ to denote the equivalence class of $(x,t)\in \overline{X^{f}}$. The {\it suspension} semi-flow over $T$ under the \textit{roof function} $f$, written by $(X^f, T^f)$, is the semi-flow $(T^f_t)_{t\in \R}$ on the space $X^f$
induced by the time translation $T_t$ on $X\times \mathbb{R}_{\ge 0}$ defined by $T_t(x,s)=(x, t+s)$. If $T$ is a homeomorphism, then $(X^f, T^f)$ defines a flow. Such (semi-)flows are regular. \\

Bowen and Walters define a metric $d_f$ compatible with the quotient topology on $X^f$  as follows  \cite[Section 4]{bowen1972expansive}. For a point $(x,t)\in X^f$ we let $u_{(x,t)}=t/f(x)$.  A  pair of points $A=(x_A,t_A)$ and $B=(x_B,t_B)$ in $X^f$ is said to be
\begin{itemize}
\item {\it horizontal} if  $u_A=u_B$, then its length is $|AB|:= (1-u_A)d(x_A,x_B)+u_Ad(Tx_A,Tx_B)$.
\item {\it vertical} if $x_A=x_B$, $x_B=Tx_A$ or $x_A=Tx_B$ , then its length is $|AB|:=| u_A-u_B|, \, 1- u_B+u_A\text{ or }1- u_A+u_B$ respectively.
\end{itemize}
A sequence $A_1,\cdots, A_n$ of $n$ points in $X^f$ is said {\it admissible} when $(A_i,A_{i+1})$ is either vertical or horizontal for $i=1,\cdots, n-1$ and the length of this sequence is defined as the sum of the length of its corresponding pairs. Then the distance $d_f$ between two points $A$ and $B$ is defined as the infimum of the length of all admissible sequences $A_1,\cdots, A_n$ with $A_1=A$ and $A_n=B$.\\

Denote by $\LL$ the Lebesgue measure on $\R$. Let $(X,T)$ be a discrete topological system. Let $f: X\to \R_{>0}$ be a continuous function and $(X^f, T^f)$ be the associated suspension semi-flow. For $\mu\in \MM_T(X)$ the product measure $\mu \times \LL$ induces a finite $T^f$-invariant measure on $X^f$, which defines a homeomorphism $\Theta$ between $\MM_{T^f}(X^f)$ and $\MM_T(X)$:
\begin{equation*}
\begin{split}
\Theta: \MM_T(X) &\to \MM_{T^f}(X^f)\\
\mu &\mapsto \frac{(\mu\times \LL)_{|X^f}}{\int f d\mu}.
\end{split}
\end{equation*}

Due to Abramov \cite{abramov1959entropy}, the entropy of $\mu$ and $\Theta(\mu)$ are related by the following formula 
\begin{equation}\label{eq:entropy suspension flow}
h_{T^f}(\Theta(\mu))=\frac{h_T(\mu)}{\int f d\mu}, \forall \mu \in \MM_T(X).
\end{equation}

\subsection{Extensions}\label{exten}

A suspension flow over a zero-dimensional  invertible dynamical system will be called a {\it zero-dimensional suspension flow} and a topological extension by a zero-dimensional suspension flow is said to be a {\it zero-dimensional extension}. Similarly a suspension flow over a symbolic discrete topological dynamical system (a.k.a. $\mathbb{Z}$-\emph{subshift}) will be called a {\it symbolic suspension flow} and a topological extension by a symbolic suspension flow is said to be a {\it symbolic extension}.\\

Let $(X, \Phi)$ and $(Y, \Psi)$ be two topological semi-flows. Suppose that $\pi: Y\to X$ is a topological extension from $(X, \Phi)$ to $(Y, \Psi)$. The topological extension is said to be
\begin{itemize}
	\item \textit{principal} when it preserves the entropy of invariant measures, i.e. $h(\mu)=h(\pi \mu)$ for all $\Psi$-invariant measure $\mu$,
	\item \textit{with an embedding} when there is a Borel embedding $\psi:(X,\Phi)\rightarrow (Y,\Psi)$ with $\pi\circ \psi =\Id_X$,
	\item \textit{isomorphic} when the map induced by $\pi$ on the sets of invariant Borel probability measures is bijective and $\pi: (Y, \Psi, \mu) \to (X, \Phi, \pi \mu)$ is a measurable isomorphism for any $\Psi$-invariant measure $\mu$,
	\item \textit{strongly isomorphic} when there is a set $E\subset X$ with $\mu(E)=1$ for all $\mu\in \mathcal M_\Phi(X)$ such that the restriction of $\pi$ to $\pi^{-1}E$ is one-to-one.
\end{itemize}
Clearly, we have the following implication:

\begin{equation*}
\xymatrix{
	\text{strongly isomorphic} \ar@2{->}[r]  \ar@2{->}[d]   &  \text{isomorphic} \ar@2{->}[d] \\
 \text{with an embedding}	 & \text{principal}. 
}
\end{equation*}

\subsection{Construction}\label{sec:Zero-dimensional principal extension of flows}
In this section we build a zero-dimensional principal extension for any topological semi-flow.
\begin{thm}\label{main thm}
	Every topological semi-flow has a zero-dimensional principal extension. Moreover the roof function of this extension may be chosen constant equal to $1$.
\end{thm}
To prove this theorem, we first show that every suspension semi-flow has a zero-dimensional principal extension (Proposition \ref{prop:zero-dimensional principal extension suspension flow}). Then, for a general topological flow we build a principal extension by  a suspension flow.


\subsubsection{Zero-dimensional principal extension of suspension flow}\label{rappel}

Let $(Z, T)$ be a discrete topological system. Let $f: Z\to \R_{>0}$ be a continuous function. In this section, we construct a  zero-dimensional principal extension of the suspension semi-flow $(Z^f, T^f)$. Due to T. Downarowicz and D. Huczek \cite{downarowiczHuczek2013zero}, there exists an invertible dynamical system $(X, S)$ which is a zero-dimensional principal extension  of $(Z,T)$. Denote by $\rho: X\to Z$ the factor map. We see that the map  $g:=f\circ \rho: X \to \R_{>0}$ is continuous. Define $\overline{\rho}: \overline{X^g} \to \overline{Z^f}$ by $\overline{\rho}(x, t)\mapsto \left(\rho(x), t\right)$ for all $(x,t)\in \overline{X^g}$. We have $\overline{\rho}(x,f(x))=(\rho(x), f\circ \rho(x))$ for all $x\in X$, so that $\overline{\rho}$ induces a continuous map $\widehat{\rho}:X^g\rightarrow Z^f$.  Moreover $\overline{\rho}$ commutes with the translation on the second coordinate, 
therefore $\widehat{\rho}$ is a continuous factor map. 

\begin{lem}
	$\widehat{\rho}$ is principal.
\end{lem}
\begin{proof}
	Let $\mu \in \MM_S(X)$. Since $\rho$ is principal, we see that $h(\rho\mu)=h(\mu)$. It is clear that
	$$
	\widehat{\rho}\Theta(\mu)=\frac{\widehat{\rho}(\mu\times \LL)_{|X^g}}{\int g \, d\mu}=\frac{(\rho\mu\times \LL)_{|Z^f}}{\int f \, d(\rho\mu)}=\Theta(\rho\mu).
	$$
	Then by \eqref{eq:entropy suspension flow}, we obtain that
	$$
	h(\widehat{\rho}\Theta(\mu))=h(\Theta(\rho\mu))=\frac{h(\rho\mu)}{\int f \, d(\rho\mu)}=\frac{h(\mu)}{\int g \, d\mu}=h(\Theta(\mu)).
	$$
	
\end{proof}

To sum up, we obtain the following proposition.
\begin{prop}\label{prop:zero-dimensional principal extension suspension flow}
	Every suspension semi-flow has a zero-dimensional principal extension.
\end{prop}
\subsubsection{General case}

We present the proof of Theorem \ref{main thm} in the general case. 
\begin{proof}[Proof of Theorem \ref{main thm}]
	Let $(X, \Phi)$ be a topological semi-flow. Let us denote by $\mathbbm{1}$  the constant function on $X$ equal to  $1$. Then the suspension semi-flow $(X^{\mathbbm 1}, (\phi_1)^{\mathbbm 1})$ over the time-$1$ map $\phi_1$ under $\mathbbm 1$ defines an extension of $(X, \Phi)$ via the factor map $\pi: (x, t)\mapsto \phi_t(x)$ for $x\in X$ and $t\ge 0$. Notice that a $(\phi_1)^{\mathbbm 1}$-invariant measure on $X^{\mathbbm 1}$ has the form $\mu\times\LL_{[0,1]}$ where $\mu$ is a $\phi_1$-invariant measure and $\LL_{[0,1]}$ is the Lebesgue measure on  $[0,1]$. Pick arbitrary $(\phi_1)^{\mathbbm 1}$-invariant measure $\mu\times \LL_{[0,1]}$. It follows from Fubini's theorem that for all Borel subset  $B$ of $X^{\mathbbm 1}$: 
	\begin{align*}
	\pi(\mu\times \LL_{[0,1]})(B)&=\mu\times \LL_{[0,1]}(\pi^{-1}((B)),\\
	&=\mu\times \LL_{[0,1]}\left(\{(x,t)\in X^{\mathbbm 1}, \ \phi_t(x)\in B\}\right),\\	
	& =\int_{0}^{1}\mu(\phi_{t}^{-1}(B))\, dt,\\
	&=\int_{0}^{1}\phi_{t}\mu (B) \, dt.
	\end{align*}
	For all $t\geq 0$, we  observe that  $h_{\phi_1}(\phi_{t}\mu)\geq  h_{\phi_1}(\mu)$. Indeed, for $s\geq 0$ with $t+s\in \mathbb N$ we have  $h_{\phi_1}(\phi_{t}\mu, \phi_s^{-1}P)=h_{\phi_1}(\mu, \phi_{s+t}^{-1}P)=h_{\phi_1}(\mu, \phi_{1}^{-(t+s)}P)= h_{\phi_1}(\mu,P)$ for any Borel finite partition $P$ of $X$. 
	  Since  the entropy function $h_{\phi_1}$ is harmonic, we get then
	\begin{align*}
	h_{\Phi}(\pi(\mu\times \LL_{[0,1]}))&=h_{\phi_1}(\pi(\mu\times \LL_{[0,1]})),\\
	&=\int_{0}^{1}h_{\phi_1}(\phi_{t}\mu)\, dt,\\
	&\geq h_{\phi_1}(\mu),\\
	&\geq h_{(\phi_1)^{\mathbbm 1}}(\mu\times \LL_{[0,1]}).
	\end{align*}

	Therefore, we conclude that $(X^{\mathbbm 1}, (\phi_1)_{\mathbbm 1})$ is a principal extension of $(X, \Phi)$. By Proposition \ref{prop:zero-dimensional principal extension suspension flow}, the suspension flow $(X^{\mathbbm 1}, (\phi_1)^{\mathbbm 1})$ has a zero-dimensional principal extension. By composition we get a principal zero-dimensional extension of $(X, \Phi)$. This completes the proof.	
\end{proof}

\section{Applications to symbolic extensions}\label{sec:Symbolic extension of flow}\label{appli}
By the previous construction of a zero-dimensional principal extension, the results related to symbolic extensions obtained in \cite{burguet2019symbolic} for flows with the small flow boundary property may be straightforwardly extended to general topological flows.
We first recall the framework of the entropy theory of M. Boyle and T. Downarowicz.

\subsection{Superenvelope of entropy structures}
Entropy structures are particular sequences of nonnegative real functions $\mathcal H= (h_k)_k$ defined on the set $\mathfrak X$ of invariant probability measures,  
$\mathfrak X=\MM_T(X)$ or $=\MM_\Phi (X)$ depending on the context,   which are converging pointwisely to the entropy function $h$. 
For a zero-dimensional system, the entropy with respect to any sequence $(P_k)_k$ of clopen partitions with $\diam(P_k)\xrightarrow{k}0$ 
defines an entropy structure. For a precise definition and examples we refer to \cite{downarowicz2011entropy} for discrete systems and to \cite{burguet2019symbolic} for topological flows. 

For a function $f:\mathfrak{X}\rightarrow \mathbb R$ we let $f^{\tilde{}}$ be the {\it upper semi-continuous envelope} of 
$f$, i.e. $$\forall \mu\in \mathfrak X, \ f^{\tilde{}}(\mu)=\limsup_{\nu\rightarrow \mu}f(\nu). $$
A {\it superenvelope} $E:\mathfrak X \rightarrow \mathbb R$ is an affine function such that for some (any) entropy structure $\mathcal H=(h_k)_k$
$$\lim_k\left(E-h_k\right)^{\tilde{}}=E-h.$$

\subsection{Characterization of  symbolic extensions}

For a symbolic extension $\pi:(Y,S)\rightarrow (X,T)$ of a topological discrete system $(X,T)$, we let 
$h^\pi$ be the associate entropy function defined as 

$$\forall \mu\in \mathcal M_T(X), \ h^\pi(\mu)=\sup_{\nu\in \mathcal M_S(Y), \, \pi\nu=\mu}h(\nu).$$

We may define in the same way the entropy function associated to the symbolic extension of a topological flow. 
Then the fundamental theorem in the theory of symbolic extension is the following characterization of the entropy function $h^\pi$ in terms of superenvelopes.

\begin{thm}[\cite{BDz2004}, Theorem 5.5]\label{thm:dicrete superenvelop}
	For a discrete topological system $(X,T)$, any function $h^\pi$ is an affine superenvelope and conversely for any affine superenvelope $E$ 	there is a symbolic extension $\pi$ with $h^\pi=E$.
\end{thm}

The first author proved the corresponding statement for topological flows with the small flow boundary property (Theorem 3.6 in \cite{burguet2019symbolic}). The existence of  a zero dimensional extension given by Theorem \ref{main thm} implies the general following version : 

\begin{thm}\label{main thm1 symbolic extension}
	For a  topological flow  $(X,\Phi)$, any function $h^\pi$ is an affine superenvelope and conversely for any affine superenvelope $E$ 	there is a symbolic extension $\pi$ with $h^\pi=E$.
\end{thm}

Similarly, we  relate the symbolic extensions of a general topological flow with the symbolic extensions of the discrete systems given by its time-$t$ maps (see Lemma 3.19 in \cite{burguet2019symbolic}).

\begin{thm}\label{main thm2 symbolic extension}
A topological flow admits a symbolic extension (resp. principal) if and only if $\phi_t$ admits a symbolic extension (resp. principal) for some (any) $t\neq 0$.
\end{thm}

We recall some  terminology of the theory of symbolic extensions, which will be used in the next sections. Firstly the (topological) symbolic extension entropy  $h_{sex}(T)$ (resp. $h_{sex}(\Phi)$) of a topological system  (resp. flow) is the infimum of the topological entropy over all its symbolic extensions. The corresponding measure theoretic quantity  is the real function $h_{sex}$ defined on the set of  invariant measures  as the infimum of the functions $h^\pi$ over all the symbolic extensions $\pi$.

\subsection{Invariance under orbit equivalence}

Two topological flows $(X, \Phi)$ and $(Z, \Psi)$ are said to be {\it orbit equivalent} when there is a homeomorphism $\Gamma$ from $X$ onto $Z$ mapping $\Phi$-orbits to $\Psi$-orbits, preserving their orientation. In other words, the topological flows $(X, \Phi)$ and $(X, \widehat{\Phi})$ with $\widehat{\Phi}=(\widehat{\phi}_t)_t=(\Gamma^{-1}\circ\phi_t\circ \Gamma)_t$ has the same orbits with the same direction, i.e. 
$$
\{\phi_t(x): t\in \R_{\ge 0} \}=\{\widehat{\phi}_t(x): t\in \R_{\ge 0} \}, \forall x\in X.$$

 Assume that $(X, \Phi)$ is regular. Then we can define the continuous map $\theta: X\times \R_{\ge 0} \to \R_{\ge 0}$ with the following properties (see \cite{beboutoff1940mesure}):
\begin{itemize}
	\item [(i)] $\phi_t(x)=\widehat{\phi}_{\theta(x,t)}(x)$,
	\item [(ii)] $\theta(x, t+s)=\theta(x, s)+\theta(\phi_s(x), t)$,
	\item [(iii)] $\theta(x,0)=0$ and $\theta(x,t)$ is strictly increasing in $t$.
\end{itemize}

	Zero and infinite   topological entropy are invariants for orbit equivalent regular flows \cite{ohno1980weak}. 
In \cite{burguet2019symbolic} the first author showed existence of symbolic extensions is preserved under orbit equivalence for topological regular flows with the small boundary property. Here by using the zero-dimensional principal extension built in Theorem \ref{main thm} we remove this extra assumption. The proof  differs completely from the one given in \cite{burguet2019symbolic} for flows with the small  flow boundary property; it  is not a direct consequence of the existence of principal zero dimensional extension as in the above Theorem \ref{main thm1 symbolic extension} and Theorem \ref{main thm2 symbolic extension}.

\begin{thm}
	The existence of symbolic extensions (resp. principal) is preserved by orbit equivalence for regular flows.
\end{thm}

\begin{proof}
	Let $\Phi$ and $\widehat{\Phi}$ be two regular topological flows  on a compact metric space $X$ with the same orbits  and let  $\theta$ be defined as above. According to Theorem \ref{main thm} the suspension flow $(X^{\mathbbm{1}}, (\phi_1)^{\mathbbm{1}})$ is a principal extension of $(X, \Phi)$ via the factor map $\pi: (x, t)\mapsto \phi_t x$. On the other hand, the function  $g:x\mapsto \theta(x,1)$ is continuous and positive on $X$ (due to (iii)). Then by the fact that $\widehat{\phi}_{\theta(x,1)}(x)=\phi_1(x)$, the suspension flow $(X^g, (\phi_1)^g)$ is the extension of $(X, \widehat{\Phi})$ via the factor map $\widehat{\pi}: (x, t)\mapsto \widehat{\phi}_tx$ :
	\begin{equation*}
	\xymatrix{
		(X^{\mathbbm{1}}, (\phi_1)^{\mathbbm{1}}) \ar[d]_{\pi}  & (X^g, (\phi_1)^g)  \ar[d]_{\widehat{\pi}} \\
		(X, \Phi) & (X, \widehat{\Phi}) 
	}
	\end{equation*}
	We claim that $(X^g, (\phi_1)^g)$ is a principal extension of $(X, \widehat{\Phi})$. To simplify the notations we denote by $\Psi=(\psi_t)_t$  the suspension flow  $(\phi_1)^g$ on $X^g$. By Ledrappier-Walters formula \footnote{Let $\pi: (X,T)\to (Y,S)$ be an extension. Then for any $S$-invariant measure $\nu$, we have that $$ \sup_{\mu: \pi \mu=\nu } h_{\mu}(T)=h_{\nu}(S)+\int_Y h_{\text{top}}(\pi^{-1}y) d\nu(y). $$ } \cite{LedWal77}, it is enough to check $h_{\text{top}}(\psi_1,\widehat{\pi}^{-1}z)=0$ for all $z\in X$. 
	Fix $z\in X$. Since
	the map $\widehat{\pi}:X^g\rightarrow X$ is given by $\widehat{\pi}(x,t)=\widehat{\phi}_tx$ for $x\in X$ and $0\leq t<g(x)$, the set
	$\widehat{\pi}^{-1}z$ is contained in the set
	$$
	\{(\widehat{\phi}_{-t}z, t ): 0\le t\le  \max g  \}.
	$$

	  We recall that $d_g$ denotes the Bowen-Walters metric  on $X^g$.
	Let $\min g>\epsilon>0$ and $N\in \mathbb N^*$. We will show the $h_{top}(\widehat{\pi}^{-1}z,c\epsilon)\leq \frac{\log 2}{N\min g}$ with $c=1+\frac{3\max g}{(\min g)^2}$, where $h_{top}(\widehat{\pi}^{-1}z,c\epsilon)$ is the Bowen  entropy of $\widehat{\pi}^{-1}z$  for $d_g$ at the scale $c\epsilon$ with respect to $\psi_1$, i.e. 
	$$h_{top}(\widehat{\pi}^{-1}z,c\epsilon)=\limsup_n\frac{1}{n}\log \max \{\sharp E_n, \ E_{n} \text{ is } (c\epsilon, n)-\text{separated for }d_g  \text{ w.r.t. }\psi_1 \}.$$
	  This will conclude the proof as $\epsilon$ and $N$ are chosen arbitrarily and $h_{top}(\widehat{\pi}^{-1}z)=\lim_{\epsilon\rightarrow 0}h_{top}(\widehat{\pi}^{-1}z,c\epsilon)$.

By uniform continuity of $g$ and the flow map, there are  $\eta, \delta>0$ which can be assumed less than $\frac{\epsilon}{2}$ such that 
$$|g(x)-g(y)|<\frac{\epsilon}{4N} \ \forall x,y \text{ with }d(x,y)<\eta$$
and 
$$d\left(\phi_{t}z,\phi_sz\right)<\eta \ \forall t,s\in \mathbb{R} \text{ with } |t-s|<\delta.$$	
We let $S_ng$ be the Birkhof sum 
$$S_ng(x)=\sum_{k=0}^{n-1}g\circ\phi_{k}(x).$$ 
Observe that 
$$
S_{n+m}g=S_ng+S_mg\circ\phi_n.
$$
We denote by $|J|$ the diameter of a subset $J$ of $\mathbb R$. From the choice of $\delta$ the set $|S_ng(\phi_Iz)|:=\{S_ng(\phi_tz), \ t\in I\}$ satisfies  $|S_ng(\phi_Iz)|<\frac{\epsilon}{4}$ for any subset  $I$ with $|I|<\delta$ and any  integer $0\leq n\leq N$. 

\vspace{10pt}

\text{\bf Claim 1. } For any subset $I$ of $\R$ with $|I|<\delta$, we can  find a cover  $\mathcal F_k=\mathcal F_k(I)$ of $I$ with $\sharp \mathcal F_k\leq 2^k$ satisfying  $|S_{n}g ( \phi_{J}z )|<\frac{\epsilon}{2}$ for all $k\in \N$, all $n\leq kN$ and $J\in \mathcal F_k$.

\vspace{10pt}

\begin{proof}[Proof of Claim 1]
	Fix such an interval $I$ with $|I|<\delta$. We will argue by induction on $k$. For $k=1$ we let $\mathcal F_1=\{I\}$. 
	Assume $\mathcal F_k$ already built and take $J\in 
	\mathcal F_k$. Then we may cover $J$ by two subsets $J^1$ and $J^2$ in such a way the diameter of  $S_{kN}g ( \phi_{J^i}z )$ is less than 
	$\frac{\epsilon}{4}$ for $i=1,2$. Since the diameter of $J^i+kN$ is less than $\delta$, we have also  $|S_{j}g ( \phi_{J^i+kN}z )|<\frac{\epsilon}{4}$ for every $0\le j\le N $ and therefore $$|S_{kN+j}g ( \phi_{J^i}z )|\le |S_{kN}g ( \phi_{J^i}z )| +|S_{j}g ( \phi_{J^i+kN}z )| < \frac{\epsilon}{2}, \forall 0\le j\le N. $$
	We conclude by letting $\mathcal F_{k+1}=\{J^1, J^2 \ : \ J\in \mathcal F_k \}$. 
\end{proof}

We go back to the proof of $h_{top}(\widehat{\pi}^{-1}z,c\epsilon)\leq \frac{\log 2}{N\min g}$. The flow $\Phi$ and $\hat \Phi$ being 
orbit equivalent, there  are finite families $\mathcal I$ and $\mathcal K$ of intervals with length respectively less than $\delta$ and $\epsilon/2$ such that $ \{(\widehat{\phi}_{-t}z, t ): 0\le t< \max g \}$ is contained in $\bigcup_{I\in \mathcal I, \, K\in \mathcal K}\phi_Iz\times K$. 


Let $m=[kN\min g]$ and let $E_m$ be a maximal $(m, c\epsilon)$-separated set inside $\widehat{\pi}^{-1}z$ for the time-one map $\psi_1$ of the suspension flow on  $(X^g, (\phi_1)^g)$.

\vspace{10pt}

\text{\bf Claim 2. } Any set of the form $\phi_Iz\times K$ with $|I|<\delta$ and $|K|<\epsilon/2$ contains at most $2^k$ points of $E_m$.

\vspace{10pt}

Suppose Claim 2 holds. Then we have
\begin{equation*}
\sharp E_m\le \sharp \mathcal{I} \sharp \mathcal{K} \cdot 2^k.
\end{equation*}  
We conclude that 
\begin{align*}
h_{top}(\widehat{\pi}^{-1}z,c\epsilon)
&=\limsup_m\frac{1}{m}\log \sharp E_m\\
&\leq \limsup_k \frac{k\log 2+\log \sharp \mathcal{I} \sharp \mathcal{K}}{kN\min g}= \frac{\log 2}{N\min 
	g}.
\end{align*}
 It remains to prove Claim 2.

\begin{proof}[Proof of Claim 2]
	Let $|I|<\delta$ and $|K|<\epsilon/2$.
	Let $x=(\phi_tz,s)$ and $x'=(\phi_{t'}z,s')$  in $\phi_Jz\times K$ for some $J\in \mathcal{F}_k(I)$. We will show that $x$ and $x'$ are $(c\epsilon,n)$-closed. This would imply that there is at most one point of $E_m$ in $\phi_Jz\times K$ and Claim 2 then follows.  By Claim 1, we have $|S_{n}g ( \phi_{J}z )|<\frac{\epsilon}{2}$ for all $n\leq kN$. For $l\leq m$ the point $x_l=(\phi_tz,s+l)$ of $X\times \mathbb R$ has for representative  in the quotient space $X^g$ the point $\left(\phi_{n+t}z, s_l \right)$ where $n$ is the largest integer with $s_l:=s+l-S_ng(\phi_tz)\geq 0$. As $l\leq kN\min g$ we have $n\leq kN$.  Similarly we may define similarly the integer $n'$ and the real number $s'_l$ associated to  $x'_l:=(\phi_{t'}z,s'+l)$. It follows from $|S_ng(\phi_Jz)|<\frac{\epsilon}{2}$, $|K|<\frac{\epsilon}{2}$ and $\epsilon<\min g$ that we have $|n-n'|\leq 1$. Let  $u_l=\frac{s_l}{g(\phi_{n+t}z)}$ and $u'_l=\frac{s'_l}{g(\phi_{n'+t'}z)}$ in $[0,1]$.  We consider the point  $x''_l$ in $X^g$ defined as $x''_l=(\phi_{n+t'}z,u_lg(\phi_{n+t'}z)$. The pairs $(x_l,x''_l)$ and $(x''_l,x'_l)
	$ are respectively horizontal and vertical. As $t,t'$ both lie in $I$ we have $|t-t'|<\delta$ so that  
\begin{align*}
|x_lx''_l|&= u_l d(\phi_{n+t}(z),\phi_{n+t'}(z))+(1-u_l)d(\phi_{n+1+t}(z),\phi_{n+1+t'}(z))\leq \epsilon.
\end{align*}	
	
Then the length $|x''_lx'_l|$ may be bounded from above  depending on $|n-n'|$ :
	\begin{itemize} 
		\item either $n=n'$. Then $|g(\phi_{n+t}z)-g(\phi_{n+t'}z)|<\epsilon$ and 
\begin{align*}
|s_l-s'_l|&\leq |s-s'|+|S_ng(\phi_tz)-S_ng(\phi_{t'}z )|,\\
&\leq \frac{\epsilon}{2}+\frac{\epsilon}{2}=\epsilon.
\end{align*}		
Therefore we get 
		\begin{align*}
		|x''_lx'_l|&= |u_l-u'_l|,\\
		&\leq \frac{ |s_l-s'_l|g(\phi_{n+t'}z)+ s_l |g(\phi_{n+t}z)-g(\phi_{n+t'}z)|}{g(\phi_{n+t'}z)g(\phi_{n+t}z)},\\
		&\leq 2\epsilon \frac{\max g}{(\min g)^2},
		\end{align*}
		
		\item or $|n-n'|=1$, say $n=n'+1$. Then $s_l<\epsilon$ and $|s_l+g(\phi_{n'+t'}z)-s'_l|<\epsilon$
		so that with the previous notations we get
		\begin{align*} 
		|x''_lx'_l|&=|u_l-u'_l+1|,\\
		&=\left|\frac{s_l}{g(\phi_{n+t}z)}-\frac{s'_l-g(\phi_{n'+t'}z)}{g(\phi_{n'+t'}z)}\right|,\\ 
        &\leq 		\left|\frac{s_l}{g(\phi_{n+t}z)}\right|+\left|\frac{s'_l-g(\phi_{n'+t'}(z))}{g(\phi_{n'+t'}z)}\right|,\\
		&\leq \frac{\epsilon}{\min g}+ \frac{2\epsilon}{\min g}\leq 3\epsilon \frac{\max g}{(\min g)^2}.
		\end{align*}
		\end{itemize}

		Consequently we get in any case 
		\begin{align*}
		d_g(x_l,x'_l)&\leq  |x_lx''_l|+|x''_lx'_l|,\\
		& \leq \epsilon +3\epsilon \frac{\max g}{(\min g)^2}=c\epsilon.
		\end{align*}
	
\end{proof}

	Assume that $(X, \Phi)$ admits a  symbolic extension. Then so does  $(X, \Phi_1)$ by Theorem \ref{main thm2 symbolic extension}. It follows from  that $(X^g, (\phi_1)^g)$ also has symbolic extension (resp. principal) by Lemma 3.18 in \cite{burguet2019symbolic}. Since  $(X^g, (\phi_1)^g)$ is a principal extension of $(X, \widehat{\Phi})$, the flow  $(X, \widehat{\Phi})$ admits a symbolic extension (resp. principal).	
\end{proof}

\section{Singular suspension flows}\label{sec:Singular suspension flows}

In Section \ref{sec:Suspension},  we have defined the suspension over a  discrete invertible topological system $(X,T)$ with a positive continuous roof function $f:X\to \R_{>0}$. 
Assume now $(X, T)$ has a fixed point $*$ and $f: X\to \mathbb{R}_{\geq 0}$ only vanishes on the  fixed point $*$.  Then we may define again the topological quotient space $X^f$ by (\ref{suspension}). We assume that $\sum_{k\in \mathbb N} f(T^kx)$ and $\sum_{k\in \mathbb N} f(T^{-k}x)$ goes uniformly to infinity out of  any neighborhood of $*$, i.e.  

\begin{align}\label{hypothesis} \forall U \text{ open with }*\in U,  \ \forall  M>0, \ \exists K\in \mathbb N, \, \text{s.t.} \nonumber \\
\forall x\in X\setminus U  \ \ \ \ \  \ \sum_{k=0}^K f(T^{k}x) , \, \sum_{k=0}^K f(T^{-k}x)>M.
\end{align}

\begin{lem} Under the hypothesis (\ref{hypothesis}) the time translation still induces a topological flow $T^f$ on $X^f$. Moreover the Bowen-Walters metric $d_f$ \footnote{We use the convention $u_{(*,0)}=0$.} is compatible with the quotient topology on $X^f$. 
\end{lem}
The flow $(X^f,T^f)$ is called a  {\it singular suspension flow}. When $(X,T)$ is a subshift we speak of {\it singular symbolic flow}.
In the following we write $*^f=(*,0)\in X^f$ the singularity of the flow.

\begin{proof}
The properties $\sum_{k\in\mathbb N} f(T^kx)=+\infty$ and  $\sum_{k\in\mathbb N} f(T^{-k}x)=+\infty$   clearly ensures the  orbit of the induced flow $T^f$ are defined on $\mathbb R$ at any $(x,t)\in X^f$ with $x\neq *$. Moreover $T^f$ is continuous on this open set $X^f\setminus \{*^f\}$. We let $T^f_t (*^f)=*^f$ for all $t\in \mathbb R$.  Let us show that for a fixed $s\in \mathbb R$, the time-$s$ map $T^f_s$ is continuous  at $*^f$.  Without loss of generality we may assume $s>0$. We argue then  by contradiction : assume there is  a neighborhood $V$ of $*^f$ in $
X^f$ such that $T^f_s(A)$ does not lie in $V$ for $A\in X^f$  arbitrarily close to $*^f$. 
For any $x\in X\setminus\{*\}$, we let $k(x)$ be the largest integer $k$  with $\sum_{l=0}^kf(T^{-l}x)< s$.
Let $U$ be an open neighborhood of $*$ with $\{ (x,t)\in X^f, \, x\in U \}\subset V $ and we 
write  $T^f_s(A)=(x^s_A,t^s_A)$. Then $x_A^s\in X\setminus U$ for $A$ arbitrarily close to $*^f$. It follows from (\ref{hypothesis}) that  $k(x^s_A)$ is bounded  
for $A$ arbitrarily close to $*^f$ by some $K>0$. Then $T^{-k(x^s_A)}x^s_A  $  does not belong to the open neighborhood $\bigcap_{k=0}^K T^{-k}U$ of $*$. But  $A=(T^{-
k(x^s_A)}x^s_A, t^s_A+s-\sum_{l=0}^{k(x^s_A)}f(T^{-l}x^s_A))$ is close to $*^f=(*,0)$. This is a contradiction.

To prove  the Bowen-Walters metric is compatible with the topology on $X^f$, in comparison with the regular case, one only needs to check that for two points $A=(x_A,t_A)$ and $B=(x_B,t_B)$ lying on the same orbit of the flow (in particular $x_B=T^k(x_A)$ for some integer $k$)   with $d_f(A,B)$ small, the minimum  $d(T^{\epsilon}x_A,x_B)$ for $\epsilon\in\{0,-1,1\}$ is also small. Indeed if $A$ is  close to $*^f$ it could  happen that the length of a  sequence $A=A_1,\cdots, A_n=B$ with $A_iA_{i+1}$ vertical for $i=1,\cdots, n-1$ is small, but $k$ so large that $x_B= T^k(x_A)$ and $x_A$ are far from each other, in particular $x_B$ would lie far form the fixed point $*$. But   (\ref{hypothesis}) prevent this case as  it  implies  $k$ is bounded.
\end{proof}

  We  let $\MM^*_T(X)$ (resp. $\MM^*_{T^f}(X^f)$) denote the convex subset (not closed) of $\MM_T(X)$ (resp. $\MM_{T^f}(X^f)$) given by measures distinct from $\delta_*$  (resp. with $*^f$ zero measure). We also let $\mathcal M^{\dagger}_{T}(X)=\MM_{T}(X)\setminus \MM^*_{T}(X)$ (resp. $\mathcal M^{\dagger}_{T^f}(X^f)=\MM_{T^f}(X^f)\setminus \MM^*_{T^f}(X^f)$). Notice that if $\mu \in \mathcal M^{\dagger}_{T^f}(X^f)$ then $\mu(*^f)>0$. Then the map
\begin{equation*}
\begin{split}
\Theta: \MM^*_T(X) &\to \MM^*_{T^f}(X^f)\\
\mu &\mapsto \frac{(\mu\times \LL)_{|X^f}}{\int f \,d\mu}.
\end{split}
\end{equation*} 
is a homeomorphism (not affine in general). Sometimes, we write $\nu_\mu=\Theta(\mu)$ for $\mu\in \MM^*_T(X)$. Moreover $\Theta^{-1}$ extends continuously to $\MM_{T^f}(X^f)$ in such a way that $\Theta^{-1}(\xi)=\delta_*$ for any $\xi\in \mathcal M^{\dagger}_{T^f}(X^f)$.

It follows from Abramov formula and the variational principle for the entropy that 

$$h_{top}(T^f)=\sup_{\mu\in  \MM^*_T(X)}h(\Theta(\mu))=\sup_{\mu\in  \MM^*_T(X)}\frac{h(\mu)}{\int f\, d\mu}.$$

\subsection{Entropy structure of zero dimensional singular suspension flow}
  For a partition $P$  of $X$ we let $A^*=A^*_P$ be the atom of $P$ containing $*$, then  $X^*=X\setminus A^*$ and $P^*$ the partition of $X^*$ induced by $P$. Fix also $\delta=\delta(P)=1/p$ with $p\in \mathbb N^*$ such that $f(x)>\frac{1}{p}$ for all $x\in X^*$.

\begin{lem}\label{ab}
 For all $ \mu\in \mathcal{M}_T(X)$ with $\mu(X^*)\neq 0$  \begin{eqnarray}\label{dcd} \frac{h(\nu_\mu,\phi_{\delta},P_{\delta})}{\delta}\geq \mu(X^*)\frac{h(\mu_{*},T_{*},P^*)}{\int f \, d\mu},
 \end{eqnarray}
 where $P_\delta$ is the partition of $X^f$ given by $P_\delta:=\{ X^f\setminus (X^* \times [0,\delta[), \ B\times [0,\delta[ \ :  \ B\in P^*\}$,  $T_{*}$ is the first return map  in $X^*$ w.r.t. $T$ and $\mu_{*}\in \mathcal M_{T^*}(X^*)$ the measure induced by $\mu$ on $X^*$. 
\end{lem}

We follow the lines of the proof of Lemma 3.8 in \cite{burguet2019symbolic}.
\begin{proof}

Let $X^*_\delta$ be the subset of $X^f$ given by $X^*_\delta= X^* \times [0,\delta[/\sim$. For the partition $P^*$ of $X^*$ we  denote by  $P^*_{\delta}=\{B\times [0,\delta[, \ B\in P^*\}$ the partition induced by $P$ on $X^*_\delta$.  We also let $R_\delta$ be the partition of $X^*_\delta$ with respect to the first return time in $X^*_\delta$ w.r.t. $\phi_\delta$ and we let  $\phi_\delta^*$ the corresponding first return map.  By applying Lemma 3.7 in  \cite{burguet2019symbolic}   to $\nu_\mu=\Theta(\mu)$, $\phi_{\delta}$  and $P^*_\delta$
we get \begin{eqnarray}\label{t}
h\left(\nu_\mu^\delta,\phi^*_{\delta},P^*_\delta\vee R_{\delta}\right)&=&\frac{h\left(\nu_\mu,\phi_{\delta},P_{\delta}\right)}{\nu_\mu(B_\delta)},\nonumber\\ 
&=&\frac{\int f\, d\mu}{\delta\mu(A_\delta)}h\left(\nu_\mu,\phi_{\delta},P_{\delta}\right).
\end{eqnarray}
But  the partition $\bigvee_{k=0}^{n-1}(\phi^*_\delta)^{-k}P^*_{\delta}$ of $X^*_\delta$ is just the partition 
$\bigvee_{k=0}^{n-1}T_{*}^{-k}P^*\times [0,\delta[$. Therefore we get, with $H_\mu(Q)=\sum_{C\in Q}-\mu(C)\log \mu(C)$ :
\begin{align*}
h\left(\nu_\mu^\delta,\phi^*_{\delta},P^*_\delta\right)&=\lim_n\frac{1}{n}H_{\nu_\mu^\delta}\left(\bigvee_{k=0}^{n-1}T_{*}^{-k}P^*\times [0,\delta[\right),\\
&= \lim_n\frac{1}{n}\sum_{C\in \bigvee_{k=0}^{n-1}T_{*}^{-k}P^*\times [0,\delta[ }-(\nu_\mu^\delta(C)\log (\nu_\mu^\delta(C),
\end{align*}
where $ \nu_\mu^\delta$ denotes the probability measure induced by $\nu_\mu$ on $X^*_{\delta}$. 
For any $B\in \bigvee_{k=0}^{n-1}T_{*}^{-k}P^*$ and $C=B\times [0,\delta[$ we have
$ \nu_\mu^\delta(C)=\frac{\nu_\mu(C)}{\nu_\mu(X^*_\delta)}=\mu(B)$.
Therefore we obtain finally :
\begin{align*}
h\left(\nu_\mu^\delta,\phi_{\delta}^*,P^*_\delta\right)&=h(\mu_*,T_{*},P^*),
\end{align*}
which together (\ref{t}) implies the required inequality (\ref{dcd}). 

\end{proof}

A subset of $X$ is said to have a {\it small boundary}  when its  boundary  has zero $\mu$-measure for any $\mu\in  \mathcal M_T(X)$.
The partition $P$ has a  small boundary, when its atoms have small boundary. We assume now $(X,T)$ has the {\it small boundary property}, i.e. $X$ has a topological base in which every element has a small boundary.  Therefore there  exists  a sequence $(P_k)_k$  of  small boundary partitions of $X$ with $\diam(P_k)\xrightarrow{k}0$. Then the sequence of entropy functions with respect to $(P_k)_k$ is an entropy structure (see \cite{Dz2005}). Let  $(\delta_k)_k=(1/p_k)_k$ be  the associated sequence of parameters $(\delta(P_k))_k$. Then $\frac{h(\nu_\mu,\phi_{\delta},(P_k)_\delta)}{\delta}=h(\nu_\mu, \phi_1, Q_k)$ with $Q_k:=\bigvee_{l=0}^{p_k-1}\phi_{\delta}^{-l}(P_k)_{\delta_k}$. The partitions $Q=Q_k$  have  also  small   boundary (w.r.t. the suspension flow) and for any  $\mu\in \mathcal M^*_T(X)$   we get  with $H(t)=-t\log t-(1-t)\log (1-t)$ for all $t\in [0,1]$:

\begin{align}\label{yield}
h(\nu_\mu, \phi_1,Q) &\geq \mu(X^*)\frac{h(\mu_{*},T_*,P^*)}{\int  f\,   d\mu}, \nonumber \\
&\geq \frac{h(\mu, T, P)-h(\mu,T, \{X^*,  A_*\} )}{\int  f\,  d\mu}, \nonumber \\
&\geq  \frac{h(\mu, T, P)-H(\mu(A_*))}{\int f \, d\mu},\\
&\xrightarrow{Q=Q_k, \ k\rightarrow +\infty} \frac{h(\mu,T)}{\int f\,   d\mu}=h(\nu_\mu, \phi_1). \nonumber 
\end{align}
From Corollary 3.2 in \cite{burguet2019symbolic} it follows that  $(h(\cdot, Q_k))_k$ defines an entropy structure of the flow $(X^f, T^f)$. 
Following \cite{Dz2005} we say for two sequences $\mathcal{G}=(g_k)_k$ and $\mathcal H=(h_k)_k$ of real functions defined on the same metric space $\mathfrak X$ that   $\mathcal H$  {\it yields} $\mathcal{G}$, written $\mathcal H\succ \mathcal G$, when  for all $\epsilon>0$ and  $\mu\in \mathfrak X$  there exists a neighborhood $U$ of $\mu$ such that $$\limsup\limits_{k\to \infty}\limsup\limits_{j\to \infty} \sup_{\mu \in U} (g_j-h_k)(\mu)\le \epsilon.$$
When $\mathfrak X$ is compact this is equivalent to  $$\limsup\limits_{k\to \infty}\limsup\limits_{j\to \infty} \sup_{\mu \in \mathfrak{X}} (g_j-h_k)(\mu)\le 0.$$  Two entropy structures $\mathcal H$ and $\mathcal G$ of a topological system or flow satisfy $\mathcal H\succ \mathcal G$ and $\mathcal G\succ \mathcal H$ (see \cite{Dz2005, burguet2019symbolic}).

 \begin{lem}\label{partial}Let $\mathcal H_T=(h_k)_k$ be an entropy structure of $T$ and $\mathcal{G}=(g_k)_k$ be the sequence of real functions $g_k$ defined on $\mathcal M_{T^f}^*(X^f)$ by $g_k=h_k\circ \Theta^{-1}$. Then  
 the restriction to $\mathcal M_{T^f}^*(X^f)$ of any entropy structure of $(X^f, T^f)$ yields $\mathcal G$.
\end{lem}

\begin{proof}
Without loss of generality we may assume $\mathcal H$ is the sequence $\left(h(\cdot, P_k)\right)_k$ with $(P_k)_k$ as above and consider the associated  entropy structure $h(\cdot, \phi_1,Q_k)$ of the suspension flow.  Fix $\nu=\Theta(\mu)\in \mathcal M_{T^f}^*(X^f)$. 
For  $j$ large enough,  $\mu(A^*_{P_j} )$ is so small that $\frac{H(\xi(A^*_{P_j} ) )}{\int f \, d\xi}\leq \epsilon$ for all $\xi$  in a neighborhood $U$ of $\mu$ in $\mathcal M_{T}^*(X)$. Then it follows from (\ref{yield}) with $\nu_\mu=\Theta(\mu)$ that 

\begin{align*}
\limsup\limits_{k\to \infty}\limsup\limits_{j\to \infty}  \sup_{\nu_\xi\in \Theta(U)}\left( h(\nu_\xi, \phi_1,Q_j)-   \frac{h(\xi, T, P_k)}{\int f \, d\xi}\right) \geq & \limsup\limits_{j\to \infty} \sup_{\xi \in U} \left( h(\nu_\xi, \phi_1,Q_j)-   \frac{h(\xi, T, P_j)}{\int f \, d\xi}\right),\\
& \geq  -\limsup_j \sup_{\xi\in U}\frac{H(\xi(A_*))}{\int f \, d\xi}\geq  -\epsilon.
\end{align*}
\end{proof}

\subsection{Singular suspension flows with small entropy at the singularity}

Under some criterion on the entropy function at the singularity, we manage to  build a symbolic extension of the suspension flow $(X^f, T^f)$ from a symbolic extension of $(X,T)$. 


\begin{prop}\label{first}
Let $(X^f,T^f)$ be a singular suspension flow. Assume the associated discrete dynamical system $(X,T)$ admits a symbolic extension $\pi$ with  $\lim_{ \mu\rightarrow \delta_*} \frac{h^\pi (\mu)}{\int f\, d\mu}=0$. 

Then the suspension flow  admits a  symbolic extension $\pi'$ with $h^{\pi'}\leq g^{har}$ 
where $g^{har}$ denotes the harmonic extension of the function  $g: \mathcal{M}_{T^f}(X^f)\rightarrow \mathbb{R}_{\ge 0}$ defined by $g( \nu)= \frac{h^\pi(\mu)}{\int f \, d\mu}$ for $ \nu =\Theta(\mu)\in \mathcal{M}_{T^f}^*(X^f)$ and $g(\nu)=0$ for others $\nu$.
\end{prop}


\begin{proof}The function $h^\pi$ is upper semi-continuous on $\mathcal M_T(X)$ and $\mu\mapsto \int f\ d\mu$ is continuous and positive on $\mathcal M^*_T(X)$. Therefore $G:\mu\mapsto \frac{h^\pi(\mu) }{ \int f\ d\mu}$ is upper semi-continuous on  $\mathcal M^*_T(X)$ and  thus so is  $g=G\circ \Theta^{-1}$  on $\mathcal{M}_{T^f}^*(X^f)$. When $\nu$ belongs to $\mathcal{M}_{T^f}^\dagger(X^f)$, we have also $\lim_{\xi\rightarrow \nu}g(\xi)\leq g(\nu)=0$. Indeed we can assume $\xi$ lies in $\mathcal{M}_{T^f}^*(X^f)$ 
because $g$ is zero on $\mathcal{M}_{T^f}^\dagger(X^f)$. Then 
\begin{align*}
\lim_{\xi\rightarrow \nu}g(\xi)&\leq \lim_{\xi\rightarrow \nu}G\circ \Theta^{-1}(\xi),\\
&\leq \lim_{\mu\rightarrow \delta_*}G(\mu)=0.
\end{align*}

 Arguing as in Lemma 3.13 in \cite{burguet2019symbolic} the function $g$ is also affine on  $\mathcal{M}_{T^f}^*(X^f)$. It is also affine on $\mathcal{M}_{T^f}^\dagger(X^f)$ as $g$ is identically zero on this set. Then if $\nu\in \mathcal{M}^{\dagger}_{T^f}(X^f)$ and $\xi \in \mathcal{M}_{T^f}^*(X^f)$, the measure $\lambda\nu+(1-\lambda)\xi $ belongs to $\mathcal{M}^{\dagger}_{T^f}(X^f)$ for all $\lambda\in ]0,1]$, therefore 
\begin{align*}
0=g(\lambda\nu+(1-\lambda)\xi)&\leq \lambda g(\nu)+(1-\lambda)g(\xi).
\end{align*}

In conclusion, the function $g$ is upper semi-continuous and convex. Its harmonic extension $g^{har}$ is also upper semi-continuous by Fact A2.20 in \cite{downarowicz2011entropy}. Observe also $g^{har}\geq h^{T^f}$. We show now $g^{har}$ is a (affine) superenvelope, which will conclude the proof of the proposition by Theorem \ref{thm:dicrete superenvelop}. By Lemma 8.2.14 in \cite{downarowicz2011entropy}, it is enough to check for some entropy structure $\mathcal H^{T^f}=(h^{T^f}_k)_k$ with harmonic functions $h_k^{T^f}$

$$\lim_k(g^{har}-h^{T^f}_k)^{\tilde{} e}=g^{har}-h^{T^f}$$
with  $f^{\tilde{} e}(\nu)=\limsup \limits_{\stackrel{\xi\rightarrow \nu}{\xi\in \mathcal{M}_{T^f}^e(X^f)}}f(\xi)$ for any real function defined on  $\mathcal{M}_{T^f}(X^f)$. 
When $\nu=\Theta(\mu)$ belongs to $\mathcal{M}_{T^f}^*(X^f)$ any ergodic measure $\xi$ going to $\nu$ also lies in $\mathcal{M}_{T^f}^*(X^f)$, so that by Lemma \ref{partial} for any entropy structure $\mathcal H^T=(h^{T}_k)_k$ we get 
\begin{align*}
\lim_k(g^{har}-h^{T^f}_k)^{\tilde{} e}(\nu)&\leq \lim_k\frac{(h^{\pi}-h_k^T)^{\tilde{} e}(\mu)}{\int f\, d\mu}, \\
&\leq \frac{h^{\pi}-h^T(\mu)}{\int f\, d\mu}= (g^{har}-h^{T^f})(\nu).
\end{align*}
Finally let  $\nu\notin \mathcal{M}_{T^f}^*(X^f)$ and $\xi_n$ be a sequence of ergodic measures going to $\nu$. Then we may write $\xi_n=\Theta(\mu_n)$ with $\mu_n\xrightarrow{n}\delta_*$ so that  
 \begin{align*}
 \lim_k(g^{har}-h^{T^f}_k)^{\tilde{} e}(\nu)&\leq (g^{har})^{\tilde{} e}(\nu), \\
 &\leq \limsup_{\mu\rightarrow \delta_*}\frac{h^\pi(\mu)}{\int f \, d\mu},\\
 &=0,\\
 &\leq (g^{har}-h^{T^f})(\nu).
 \end{align*}

\end{proof}



\begin{cor}\label{singprin}
Assume $(X,T)$ admits a principal symbolic extension and $\lim_{ \mu\rightarrow \delta_*} \frac{h(\mu)}{\int f\, d\mu}=0$, then the suspension flow $(X^f, T^f)$ also admits a principal symbolic extension. In particular, if $(X,T)$ has topological entropy zero, then $(X^f, T^f)$ admits a  symbolic extension with zero topological entropy.
\end{cor}

\subsection{Noninvariance under orbit equivalence}
Contrarily to regular flows, two  singular flows with finite topological entropy  may be orbit equivalent but one admitting a symbolic extension and not the other. 

\begin{prop}\label{equicc}
There are two orbit equivalent (singular) flows $(X,\Phi)$ and $(X,\Psi)$ with $h_{top}(\Phi)=h_{top}(\Psi)<+\infty$ such that $(X,\Phi)$ admits  a symbolic extension but not $(X,\Psi)$.
\end{prop}

\begin{proof}
For any  $n\in \mathbb N$ there exists a topological system $(X_n,T_n)$ with $h_{top}(T_n)=4^{-n}$ and $h_{sex}(T_n)=3^{-n}$ (See Theorem D.1 in \cite{BFF02}).  We may let $(X_0,T_0)$ be a subshift. 
Let $X=\coprod_{n\in\mathbb N}X_n\cup\{*\}$ be the one point compactification of the $X_n$'s. We let $T:X \circlearrowleft$ defined by $T|_{X_n}=T_n$ and $T*=*$. The symbolic extension entropy $h_{sex}(T)$ of $(X,T)$ satisfies $h_{sex}(T)=\sup_nh_{sex}(T_n)=1$.  Then we consider the roof functions $f$ and $f'$ with $f(*)=f'(*)=*$  and $f|_{X_n}=\frac{1}{2^n}$, $f'|_{X_n}=\frac{1}{4^n}$ for all integers $n$. The  hypothesis (\ref{hypothesis}) is easily checked in this case. The topological entropies of the associated singular suspended flow $\Phi_f$ and $\Phi_{f'}$ satisfy $h_{top}(T^f)=h_{top}(\Phi_{f'})=h_{top}(T_0)=1$. Moreover we have  $$h_{sex}(T^{f'})\geq \frac{h_{sex}(T_n)}{f'|_{X_n}}=(4/3)^n$$ for all $n\in \N$ so that $\Phi_{f'}$ does not admit any symbolic extension. We check now $h_{sex}(T^{f})=1$.

For any $n$ there is a symbolic extension of the time-$1$ map of $(T^f)|_{X_n^f}$ with topological entropy less than $(3/4)^n$. We let $
\underline{E_n}:\mathcal M(X_n^f,(T^f)_1 )\rightarrow \mathbb R$ be the associated superenvelope given by the entropy function of this 
extension. Each $\underline{E_n}$ may be extended to an affine upper semi-continuous function $E_n$ on $\mathcal M(X^f,(T^f)_1)$ with 
$E_n(\mu)=0$ for $\mu(X_n^f)=0$. As $\underline{E_n}$ and thus $E_n$ is bounded from above by $(3/4)^n$, the function $E:=\sum_n 
E_n$ defines an affine upper  semi-continuous function. Let $\mathcal H =(h_k)$ be an entropy strucutre of $\left(X^f, (T^f)_1\right)$. By 
using again  Lemma 8.2.14 in \cite{downarowicz2011entropy}, to show $E$ is superenvelope it is enough to check for all $\mu$ in the 
closure of ergodic measures and  for all $\epsilon>0$, there exists $k$ such that 
$$\limsup_{\nu\rightarrow \mu, \ \nu \text{ ergodic}}(E-h_k)(\nu) \leq (E-h)(\mu)+\epsilon.$$ 
Either $\mu$ is supported on some $X_n$, then so does any ergodic measure $\nu$ close enough to $\mu$  and we may conclude in this case since $E(\mu)=E_n(\mu)$, $E(\nu)=E_n(\nu)$ and $E_n$ is a superenvelope of the time-$1$ map of the flow on $(X_n)_f$. Or $\mu$ is the Dirac mass at $*$ and $\limsup_{\nu\rightarrow \mu, \ \nu \text{ ergodic}}(E-h_k)(\nu)\leq \limsup_n\sup_{\xi}E_n(\xi)=0=E(\delta_*)$. Therefore $E$ is an affine superenvelope of the time-$1$ map of $T^f$, so that by Theorem \ref{main thm} the suspension flow $T^f$ admits a symbolic extension with topological entropy equal to $\sup_\mu E(\mu)=1$. 
\end{proof}

\begin{rem}In the above example, the discrete system $(X,T)$  admits a symbolic extension, however the singular suspension flow $(X^{f'}, T^{f'})$ has finite topological entropy, but  no  symbolic extensions.
\end{rem}

\subsection{Universal symbolic suspension flow}

In the following we consider  the two full shift with the singularity given by the infinite zero sequence, i.e. $(X,T)=\left(\{0,1\}^{\mathbb Z}, \sigma\right)$ and $*=0^\infty$. We let $d$ be the usual distance on $X$ given by
$$d((u_n)_n, (v_n)_n)=2^{-\min\{|n|, \ u_n\neq v_n \}}.$$
In particular we have $d((u_n)_n, *)=2^{-k_u}$ with $k_u=\min\{|n|, \ u_n=1 \}$.

We investigate the entropy and symbolic extension of a  singular symbolic suspension flow over $(X,T)$ associated to a roof function of the form $f(x)=R(d(x,*))$ for some continuous function $R:[0,1]\rightarrow \mathbb R_{\ge 0}$ with $R(*)=0$ and $R(x)>0$ for $x\neq *$. We may also write $f$ as $f(u)=g(k_u)$ for $u\not=0$ and for a function  $g:\mathbb N\rightarrow \mathbb R_{\ge 0}$ with $\lim_{k\rightarrow +\infty} g(k)=0$ by letting $g=R(2^{-\cdot})$.
With these notations the singular symbolic flow $X^f$ is well defined (i.e. (\ref{hypothesis}) holds true)  if and only if $\sum_{k\in \mathbb N}g(k)=+\infty$.

\begin{thm}\label{singular}
Assume the sequence $(kg(k))_{k\in \mathbb N}$ is converging to $l\in \mathbb R^+\cup\{+\infty\}$ when $k$ goes to infinity. Then 
\begin{enumerate}
\item if $l=0$, the topological entropy of $(X^f,T^f)$ is infinite, 
\item if $0<l<+\infty $,  then the symbolic extension entropy of $(X^f, T^f)$  satisfies 
$$\forall \nu\in \mathcal M_{T^f}(X^f), \ h_{sex}(\nu)=h^{\tilde{}}(\nu)=h(\nu)+\frac{1}{2l}\nu(*^f),$$
\item if $l=+\infty$  the flow $(X^f,T^f)$ admits a principal symbolic extension. 
\end{enumerate}
\end{thm}

\begin{ques}Does any singular symbolic suspension flow  with finite topological entropy  admit  a symbolic extension?
\end{ques}

 As a consequence of Theorem \ref{singular}, we present the following examples. 
Let $g(k)=1/k$, $g'(k)=1/\sqrt k$ and $g''(k)=1/k\log k$ for all $k$. The associated suspension flows $(X^f, T^f)$, $(X^{f'}, T^{f'})$ and $(X^{f''}, T^{f''})$ are orbit equivalent. Since $kg(k) \to 1$ and $kg'(k)\to +\infty$ as $k\to +\infty$, it follows from Theorem \ref{singular} that $(X^{f'}, T^{f'})$ admits a principal symbolic extension contrarily to $(X^f,T^f)$. Moreover they have both finite topological entropy, but $h_{top}(T^{f''})=+\infty$.  There exist also  orbit equivalent singular  topological flows $(X,\Phi)$ and $(X,\Psi)$ with $h_{top}(\Phi)=0$ and $h_{top}(\Psi)=+\infty$ \cite{sun1}.

\begin{cor}\label{equic}There are two orbit equivalent singular symbolic  flows $(X,\Phi)$ and $(X,\Psi)$ with $h_{top}(\Phi)=h_{top}(\Psi)<+\infty$ such that $(X,\Phi)$ admits  a principal  symbolic extension but not $(X,\Psi)$.
\end{cor}

Corollary \ref{equic} and Proposition \ref{equicc} raise the following question :

\begin{ques}
 Do there exist orbit equivalent singular topological flows with equal topological entropy one without any symbolic extension and the other with a principal one?
\end{ques}

Now we present the proof of Theorem \ref{singular}.
\begin{proof}[Proof of Theorem \ref{singular}] (1) We first show that \begin{equation}\label{bern}\limsup_{\nu\rightarrow \delta_{*^f}}h(\nu)\geq \frac{1}{2l}.\end{equation}
This clearly implies 
the first item. Moreover the symbolic extension entropy function $h_{sex}$ is  upper semi-continuous and concave, therefore writing $\nu \in \mathcal M_{T^f}(X^f)$ as $\nu=\nu(*^f)\delta_{*^f}+(1-\nu(*^f))\xi$ with $\xi \in \mathcal M_{T^f}^*(X^f)$ we get a first inequality in the second item
\begin{align*}
h_{sex}(\nu)&\geq  (1-\nu(*^f))h_{sex}(\xi)+\nu(*^f)h_{sex}(\delta_{*^f}),\\
& \geq (1-\nu(*^f))h(\xi) +\nu(*^f)\limsup_{\eta\rightarrow \delta_{*^f}}h(\eta),\\
&\geq  h(\nu)+\frac{1}{2l}\nu(*^f).
\end{align*}

To prove (\ref{bern}) we  consider the Bernoulli measure $\mu_\lambda\in \mathcal M_\sigma(\{0,1\}^{\mathbb Z})$ with parameter $\lambda=\mu([1])$ where $[1]$ denotes the cylinder $[1]:=\{(u_n)_n\in \{0,1\}^\mathbb Z: \ u_0=1\}$. We also denote by $[0^{2k+1}]$ for $k\in \mathbb N$ the cylinder $[0^{2k+1}]:=\{(u_n)_n\in \{0,1\}^\mathbb Z: \ u_m=0 \text{ for }|m|\leq k \}$. We compute for $\lambda$ close to $0$ :
\begin{align*}
\int f\, d\mu_\lambda&=\mu_\lambda([1])+\sum_{k\in \mathbb N}g(k)\mu_\lambda \left( [0^{2k+1}]\setminus [0^{2(k+1)+1}]\right),\\
& = \lambda+l\sum_{k\in \mathbb N}\frac{\lambda(2-\lambda)(1-\lambda)^{2k+1}} {k}+O(\lambda),\\
& = \lambda- l\lambda(1-\lambda)(2-\lambda) \log(\lambda(2-\lambda))+O(\lambda),\\
& =-2l\lambda \log\lambda+O(\lambda),
\end{align*} 
where $\limsup\limits_{\lambda\to 0} \left|\frac{O(\lambda)}{\lambda}\right|<+\infty$.
As the entropy of $\mu_\lambda$ is equal to $-\lambda\log \lambda-(1-\lambda)\log (1-\lambda)=-\lambda\log \lambda+O(\lambda)$, we get 
\begin{align*}
h(\nu_{\mu_\lambda} )=\frac{h(\mu_\lambda)}{\int f\, d\mu_\lambda }&=\frac{-\lambda\log \lambda+O(\lambda)}{-2l\lambda\log \lambda+O(\lambda)}, \\
&\xrightarrow{\lambda\rightarrow 0}\frac{1}{2l}. 
\end{align*}
Moreover the $T^f$-invariant measure $\nu_{\mu_\lambda}$ is going to the Dirac measure  at the singularity when $\lambda$ goes to zero. Indeed for all fixed $k$, we have for some polynomial $P_k$ with $P_k(0)\neq 0$: 
\begin{align*}
\nu_{\mu_\lambda} ([0^{2k+1}]\times \mathbb R/\sim)&=1-\frac{\int_{X\setminus [0^{2k+1}]} f\,d\mu_\lambda}{\int_X f\, d\mu_\lambda},\\
&\geq 1-\frac{\lambda P_k(\lambda) +O(\lambda)}{\int_X f\, d\mu_\lambda},\\
&\xrightarrow{\lambda\rightarrow 0}1.
\end{align*}

\vspace{0,2cm}

(3) Let us show now the last item assuming the second one. Assume $kg(k)$ goes to infinity as $k$ tends to infinity. For $a>0$, let $g_a(k)=\min (g(k), \frac{a}{k})$ for all $k$. The associated roof function $f_a$ satisfies the hypothesis of the second item, that is, $\lim\limits_{k\to \infty}kg_a(k)=a$. The symbolic extension entropy function is upper semi-continuous, therefore $h^{\tilde{}}\leq h_{sex}$  and  we get for all $a>0$ :
\begin{align*}\limsup_{\mu\rightarrow \delta_*}\frac{h(\mu)}{\int f\, d\mu}& \leq \limsup_{\mu\rightarrow \delta_*}\frac{h(\mu)}{\int f_a\, d\mu},\\
&\leq h^{\tilde{}}(\delta_{*^{f_a}}),\\
& \leq h_{sex}(\delta_{*^{f_a}})= \frac{1}{2a}.
\end{align*}
By letting $a$ go to $+\infty$, we get $\limsup_{\mu\rightarrow \delta_*}\frac{h(\mu)}{\int f\, d\mu}=0$. By applying Corollary \ref{singprin}, the singular suspension flow $(X^{f},T^f)$ admits a principal symbolic extension. 

\vspace{0,2cm}

(2) We prove now the second item. For $g$ with $\lim_k kg(k)= l$ we build a symbolic extension of $(X^f, T^f)$ with 
$h^{\pi}(\nu)\leq h(\nu)+\frac{\nu(*^f)}{2l}$ for all $\nu\in \mathcal M(X^f, T^f)$. This will conclude the proof of Theorem  \ref{singular}.\\

\underline{{\it Step 1 :} Construction of a principal extension by a regular suspension flow $\pi:(Y^{f'},S^{f'})\rightarrow (X^f, T^f)$.}
We first build a principal extension of $(X^f, T^f)$ by a regular suspension flow $(Y^{f'},S^{f'})$. For $x\in \{0,1\}^{\mathbb Z}$ we let $k^+_x=\min\{n\in \mathbb N^*, \ x_n=1\}$ et $k_x^-=\min \{n\in \mathbb N, \ x_{-n}=1\}$ and we 
denote by $\mathbf k_x$ the pair $(k^-_x,k^+_x)$.  We consider the partition of  $\mathbb E:=
\left( \mathbb{N} \cup \{+\infty\}\right)\times \left( \mathbb{N}^{}\cup \{+\infty\} \right) \setminus 
\{(+\infty, +\infty)\} $ into the following subsets of points $\mathbf k=(k^-,k^+)\in  \mathbb E$:
\begin{itemize}
\item $\mathcal R_1:=\{   k^-=0\}$,
\item $\mathcal R_2:=\{ 0<k^-\leq \frac{1}{3}k^+\}$,
\item $\mathcal R_3:=\{ k^+>k^-> \frac{1}{3}k^+>0\}$,
\item  $\mathcal R_4:=\{  0<k^+\leq k^-\}$.
\end{itemize}
We let  $L: \mathbb E \rightarrow \mathbb N$ be  the function satisfying for all $\mathbf k=(k^-,k^+)\in \mathbb E$ :

\begin{enumerate}
\item $L(\mathbf k)=1$ for $\mathbf k\in \mathcal R_1$, 
\item $L(\mathbf k)=k^-$ for $\mathbf k\in \mathcal R_2$, 
\item $L(\mathbf k)=k^+-\lfloor \frac{(k^-+k^+)^2}{8k^-} \rfloor $ for $\mathbf k\in \mathcal R_3$,
\item $L(\mathbf k)=\lceil k^+/2\rceil$  for $\mathbf k\in \mathcal R_4$.
\end{enumerate}

\begin{lem}\label{injec}
 For all  $\mathbf k=(k^-,k^+)\in \mathcal R_3$, we have with $R(p,q)=\sum_{j=p}^q\frac{1}{j}$  for positive integers $q>p$:
\begin{enumerate}
\item $L(\mathbf k)> \frac{k^+-k^-}{2}\geq 0$,
\item $L(\mathbf k)\leq \lceil \frac{k^+}{2} \rceil$,
\item $\left| k^++k^--\lceil \sqrt{8k^-(k^+-L(\mathbf k))}\rceil \right|\leq 4$,
\item $R\left(k^-,\lfloor \frac{k^++k^-}{2}\rfloor\right)+R\left(k^+-L(\mathbf k),\lceil \frac{k^++k^-}{2} \rceil\right)\xrightarrow{k^-\rightarrow +\infty}\log 2$.
\end{enumerate}
\end{lem}

\begin{proof}
One checks easily the two first items. Let us just show the two last ones. 
From the definition of $L$ on $\mathcal R_3$ we have  by using $\sqrt{1-x}\geq 1-x$ for all $x\in [0,1]$ :
\begin{align*}
\sqrt{8k^-\lfloor \frac{(k^-+k^+)^2}{8k^-} \rfloor }&=\sqrt{8k^-(k^+-L(\mathbf k))}\leq  k^++k^- ,\\
\sqrt{(k^-+k^+)^2-8k^-}&\leq \sqrt{8k^-(k^+-L(\mathbf k))}\leq  k^++k^- ,\\
 (k^++k^-)\left(1-\frac{8k^-}{(k^-+k^+)^2}\right)&\leq  \sqrt{8k^-(k^+-L(\mathbf k))}\leq k^++k^- ,\\
k^++k^- -4&\leq \sqrt{8k^-(k^+-L(\mathbf k))}\leq k^++k^-.
\end{align*}
The second line makes sense only if $8k^-\leq (k^++k^-)^2$. But $8k^-> (k^++k^-)^2$ implies $k^{-}\leq 1$ and  $k^+\leq 2$. In this remaining case we have therefore again $\left| k^++k^--\lceil \sqrt{8k^-(k^+-L(\mathbf k))}\rceil \right|\leq 4$.

Finally we prove item (4). Observe that $k^+$ and $k^-$ are going simultaneously to infinity for $\mathbf k\in \mathcal R_3$. Then as $R(p,q)= \log \frac{q}{p} +\frac{o(p)}{p}$ for large $p$ where $\lim\limits_{p\to \infty} \frac{o(p)}{p}=0$, we get 

\begin{align*}
R\left(k^-,\lfloor \frac{k^++k^-}{2}\rfloor\right)+R\left(k^+-L(\mathbf k),\lceil \frac{k^++k^-}{2} \rceil\right)&= \log \left(\frac{(k^++k^-)^2}{4k^-(k^+-L(\mathbf k))} \right)+\frac{o(k^-)}{k^-},\\
&\xrightarrow{k^-\rightarrow +\infty}\log 2.
\end{align*}

\end{proof}
Let $S:\{0,1\}^{\mathbb Z}\circlearrowleft$ be the topological system defined as  $Sx=\sigma^{L(\mathbf k_x)} x$ for $x\neq *$ and $S*=*$. 
The map $S$ is continuous at any $x\neq *$  because in this case $L(\mathbf k_y)$ is constant for $y$ in  some  neighborhood of $x$. Now if $x$ is close to the zero sequence then $k^+_x$ and $k^-_x$ are  large. Moreover we have always $0\leq  L(\mathbf k_x) \leq \lceil \frac{k^+_x}{2} \rceil$ and therefore $y=\sigma^{l(\mathbf k_x)}x$, which satisfies $y_n=0$ for  $n=-k^{-}_x, \cdots, \lfloor \frac{k^+_x}{2} \rfloor$, is also close to the zero sequence. For $x\in X$ with $x_0=1$, we let $p_x=\min \{n\in \mathbb N^*,   \ (S^n x)_0=1 \}$.


\begin{lem}\label{fr}
For all $x\in X$ with $x_0=1$ and 
$3\leq k_x^+<+\infty$  there exists a unique $0<r<p_x$ such that 
\begin{itemize}
 \item $\mathbf k_{S^qx}\in \mathcal R_2$ for $1\leq q<r$, 
 \item  $\mathbf k_{S^rx}\in \mathcal R_3$, 
 \item $\mathbf k_{S^qx}\in \mathcal R_4$ for $r< q\leq p_x-1$.
 \end{itemize}
 
Moreover 

\begin{enumerate}
\item  for $1\leq q\leq r$ we have $$k^-_{S^qx}=2^{q-1},$$
\item there exists a sequence $\epsilon_{r+1},\cdots, \epsilon_{p_x-2}\in \{0,1\}$ depending 
only on $k^+_x$ such that  for $ r<q<p_x$ we have $$ k^+_{S^{q}x}=2^{p_x-1-q}+
\sum_{0\leq i< p_x-q-1} \epsilon_{q+i} 2^{i}.$$
\item $p_x-2-r\leq \lceil \frac{p_x}{2}\rceil -1$.
\end{enumerate}

\end{lem}

\begin{proof}

We only prove the two last items (2) and (3), as the other conclusion is easily checked. 
  For $p_x-1>q>r$ 
 we have $k^{+}_{S^{q+1}x}=k^{+}_{S^{q}x}-\lceil\frac{k^{+}_{S^{q}x}}{2} \rceil$, therefore $k^{+}_{T^{q}
 x}=2k^{+}_{S^{q+1}x}+\epsilon_q$ with $\epsilon_q\in \{0,1\}$ depending on the parity of $k^{+}_{S^{q}x}$. Also $k^{+}_{S^{p_x-1}x}=1$. Then we get by a direct induction the desired formula for  $k^+_{T^{q}x}$.
 
Concerning the last item, we have $\mathbf k_{S^rx}\in \mathcal R_3$, so that $k^+_{S^ry}>k^-_{S^ry}=2^{r-1}>\frac{k^+_{S^ry}}{3}$. Then we have $2^{p_x-r-2}\leq k^+_{S^{r+1}x}\leq  k^+_{S^ry}\leq 3\cdot 2^{r-1}$, thus $2^{p_x}\leq 2^{2r+3}$, i.e. $r\geq \frac{p_x-3}{2}$.  
  
\end{proof}

\begin{rem}
The integers $p_x$, $r$ and the sequence $\epsilon_{r+1},\cdots, \epsilon_{p_x-2}\in \{0,1\}$ only depend on $k_x^+$ for $x\in X$ satisfying  $3\leq k_x^+<+\infty$.
\end{rem}

Now we consider the roof function $f'$ on $X$ given by
$$\forall x\in X\setminus \{*\}, \ f'(x)=\sum_{k=0}^{ L(\mathbf k_x)-1 }f(\sigma^kx) $$
$$f'(*)=l\log 2.$$

\begin{lem}
$f'$ is continuous  on $X$.
\end{lem} 
 
 \begin{proof}
Clearly it is enough to check the continuity at $*$, i.e. when $k_x^-$ and $k^+_x$ both go to infinity. 
We have 
\begin{itemize}
\item  $f'(x)\sim lR(k_x^-,2k_x^-)$ when $k_x$ lies in $\mathcal R_2$,
\item    $f'(x)\sim l R\left(k_x^-,\lfloor \frac{k_x^++k_x^-}{2}\rfloor\right)+l R\left(k_x^+-L(\mathbf k_x),\lceil \frac{k_x^++k_x^-}{2} \rceil\right)$, when  $k_x$ lies in $\mathcal R_3$,
\item $f'(x)\sim l R(\lfloor\frac{k^+_x}{2}\rfloor,k_x^+)$, when  $k_x$ lies in $\mathcal R_4$.
\end{itemize}
In all cases (see Lemma \ref{injec} (4) for the second case) we get $f'(x)\xrightarrow{x\rightarrow *}l\log 2$.

 \end{proof}

 Let $Y=\bigcap_{n\in \mathbb N}S^nX$. Any sequence $x=(x_n)_n\in X=\{0,1\}^{\mathbb Z}$ with $x_n=0$ for $n\leq  0$ or with $x_0=1$ belongs to $Y$. In particular for  any  $x\in X$ there is $k>0$ with $\sigma^{-k}x\in Y$. We let  $(Y^{f'}, S^{f'})$ be the suspension flow over $(Y, S|_Y)$ with roof function $f'$.
 
  We also denote by $Z$ the subset of $X$ given by sequences with infinitely many $1$'s in the future and in the past. The set of recurrent points of $(Y, S|_Y)$ is given by $(Y\cap Z) \cup\{*\}$. 
  
\begin{lem}
The map 
$
\pi:(Y^{f'}, S^{f'})\rightarrow  (X^f, T^f)$, $ 
 (x,t)\mapsto T^f_{t}(x,0)
$ is a principal topological extension.
\end{lem}

\begin{proof}
To prove the extension property, it is enough to see $\pi(x, f'(x))=\pi(Sx,0)$, i.e. $T^{f}_{f'(x)}(x,0)=(Sx, 0)$ for all $x\in Y$ which follows from 
\begin{align*}T^f_{f'(x)}(x,0)&=T^f_{\sum_{l=0}^{ L(\mathbf k_x)-1 }f(\sigma^kx)}(x,0),\\
&=T^f_{f(\sigma^{ L(\mathbf k_x)-1}x)}\circ \cdots \circ T_{f(x)}^f(x, 0),\\
&=(\underbrace{\sigma\circ \cdots \circ \sigma}_{L(\mathbf k_x) \text{ times}}(x), 0),\\
&=(Sx, 0).
\end{align*}

The factor map $\pi$ is surjective. Indeed for  any  $x\in \{0,1\}^{\mathbb Z}$ there is $k>0$ with $\sigma^{-k}x=y\in Y$. Therefore for  $(x,t)\in X^f$ there is $s\geq 0$ with  $\pi(y,s)=T^f_s(y,0)=(x,t)$.
The recurrent points different from $*^f$ in $X^f$ are contained in the subset 
$Z\times \mathbb{R} / \sim$ of $X^f$. For 
$x\in Z$ we let $y=\sigma^{-k}x\in  Y$ where  $k$ is the smallest nonnegative integer  with $\sigma^{-k}x\in Y$. Then   $\pi^{-1}(x,t)=\{(y,s)\}$ for some $0\leq s\leq f'(y)$. Moreover $\pi^{-1}*^f$ is contained  in the subset $\{0\}\times \mathbb{R} / \sim$ of 
$Y^{f'}$ and the suspension flow $T^{f'}$ restricted to this set is topologically conjugated to the translation flow on the circle. In any case we have $h_{top}(\pi^{-1}(x,t))=0$ for all recurrent points $(x,t)$ of $(X^f, T^f)$. By Ledrappier-Walters formula, $\pi$ is a principal extension.  
\end{proof}

\underline{{\it Step 2 :} A symbolic extension $\chi$ of $(Y,S)$.} 
To conclude the proof of Theorem \ref{singular} we investigate the symbolic extensions of $(Y,S)$. We will build a symbolic extension $\chi$ of $(Y,S)$ (with an embedding) with entropy function $h^\chi$ equal to $\mu\mapsto h(\mu)+\frac{\mu(*)\log 2}{2}$.

 Let us call a {\it block} any finite word  of the form $10^l:=1\underbrace{0 \cdots 0}_{l \text{ times}}$ with $l\in \mathbb N$. Any sequence in 
$Y\cap Z$ is an infinite concatenation of such blocks. A map $\psi:Y\cap Z\rightarrow \mathcal A^\mathbb{Z}$ for some alphabet $\mathcal A$ is said a {\it block code map }if for any $x\in Z\cap Y$ whose orbit is given by the concatenation of blocks $(B_n)_{n\in \mathbb Z}$ the orbit of $\psi(x)$ under the shift on $ \mathcal A^\mathbb{Z}$  is the concatenation of $(\Psi(B_n))_n$ for some map $\Psi$ from the set of all blocks to $\bigcup_{n\in \mathbb N}\mathcal A^n$. For $\Psi$ given, the map $\psi$  will be completely defined by letting the zero coordinate of $\psi(x)$ be equal to the zero-coordinate of $\Psi(B_0)$  for any $x\in Y\cap Z$ with $x_0=1$   and by ensuring $\psi\circ S=\sigma\circ \psi$ where $\sigma$ denotes here the shift on $\mathcal A^\mathbb{Z}$.

We define now the map $\Psi$. The alphabet $\mathcal A$ is given by $$\mathcal A =\{y^z, \ y\in \{1,2,3,4\} \text{ and }z\in \{1,2,3,4,\times\}\}.$$   For $y^z\in \mathcal A$ we will refer to $y$ as the $y$-coordinate of $y^z$.  For any $l\in \mathbb N$, we let   $\Psi(10^l)=y_1^{z_1}\cdots y_p^{z_p}$ be the word  of length $p=p_x$ for any $x\in Y$ with $x_0=x_{l+1}=1$ and $x_1=\cdots =x_l=0$ such that we have $y_q=i$ whenever $\mathbf k_{S^{q-1}x}\in \mathcal R_i$ for any $1\leq q\leq p$. Finally  we put  with the notations of Lemma \ref{fr}:
\begin{itemize}
\item $z_1=\left|k^++k^--\lceil \sqrt{8k^-(k^+-L(\mathbf k))}\rceil \right|$ with $(k^-,k^+)=\mathbf k_{T^rx}$, then $0\leq z_1\leq 4$ by Lemma \ref{injec}  (3), 
\item $z_{p-2i}=\epsilon_{p-2-i}$ for $0\leq i\leq p-3-r$ (note that $p-2( p-3-r)>1$ by Lemma \ref{fr} (3)).
\item $z_i=\times $ for others $i$. 
\end{itemize}
\begin{lem}
The block code map $\psi$ defines a uniform generator of $(Y,S)$, i.e. $\psi$ is a Borel embedding of $(Y\cap Z,S)$ to  $(\mathcal A^\mathbb{Z},\sigma)$ such that $\chi:=\psi^{-1}$ extends in a symbolic extension of $(Y,S)$ on $\overline{\psi(Y\cap Z)}\subset \mathcal A^\mathbb{Z}$.
\end{lem}


\begin{proof}

We first  check $\Psi$ is injective on the set of blocks. Assume $\Psi(10^l)=y_1^{z_1}\cdots y_p^{z_p}$. Then we may recover $l$ from the $y_i$ and $z_i$ as follows. For $x\in Y$ with $x_0=x_{l+1}=1$ and $x_1=\cdots =x_l=0$ we have
\begin{itemize}
\item $r=\sharp \{i, \, y_i=2\}+1$,
 \item $z_{p-2i}=\epsilon_{p-2-i}$ for $0\leq i\leq p-3-r$.
\item $k^{-}_{S^rx}=2^{r-1}$,
\item $k^+_{S^{r+1}x}=2^{p-r-2}+
\sum_{0\leq i< p-r-2} \epsilon_{r+1+i}2^{i}$,
\item $l=k^{-}_{S^rx}+k^+_{S^{r}x}=z_1+\lceil \sqrt{8k^-_{S^{r}x}k^+_{S^{r+1}x}}\rceil$.
\end{itemize}
Then we may rebuild  the decomposition of $x\in Y\cap Z$ into blocks which are delimited by the $y$-coordinates equal to  $1$ in the sequence $\psi(x)$. Therefore, to prove the injectivity of $\psi$ on $Y\cap Z$, it only remains to show how to repair the position of $x$ inside the blocks, that is $k^{-}_x$ or $k^+_x$. 
With the above notations, assume $(\psi(x))_0=y_q^{z_q}$. We have :
\begin{itemize}
\item if $y_q=1$ then $k^{-}_x=0 $
\item if $y_q=2$ or $3$, we get $k^{-}_x=2^{q-2}$,
\item if $y_q =4$, then   $k^+_{x}=2^{p-1-l}+
\sum_{0\leq i< p-q-1} \epsilon_{q+i} 2^{i}$.
\end{itemize}

Let us check now that $\psi^{-1}$ extends to a symbolic extension of $(Y,S)$. 
 Any $u\in \overline{\psi (Y\cap Z)}\setminus \psi(Y\cap Z)$ may be written as the concatenation of words of the form $\Psi(B)$ for blocks $B$ and semi-infinite or bi-infinite words in the alphabet $\mathcal  A$ without any $y$-coordinate equal to  $1$.
 In the last case, when there is no $y$-coordinate equal to $1$ in $u$,  we just let $\chi(u)$ be the zero sequence, i.e. $\chi(u)=*$.  Moreover  a  semi-infinite block without $y$-coordinates equal to $1$ is sent to a semi-infinite sequence of $0$'s. Then the position of $x=\chi(u)$ inside such a semi-infinite sequence is  defined as below :
 \begin{itemize}
 \item if the semi-infinite sequence lies in the future then  $k^{-}_x=2^{l-2}$, 
 \item if it lies in the past $k^+_{x}=2^{p-1-l}+
\sum_{0\leq i< p-l-1} \epsilon_{l+i} 2^{i}$ where $l$ is the position of $u$ inside the semi-infinite block. 
\end{itemize}
Defined in this way, the extension $\chi$ of $\psi^{-1}$ on   $\overline{\psi (Y\cap Z)}$ is continuous. 
 Observe finally  that $\chi$ is surjective on $Y$ because  the image of $\chi$ contains $Y\cap Z$ which is dense in $Y$.

\end{proof}

\begin{rem}By letting $\psi(*)$ be the $1^\times$ sequence $(1^\times)^\infty$, we get a Borel embedding of the set of recurrent points of $(Y,S)$ to $\mathcal A^\mathbb Z$. By Remark 1.4 in \cite{bd2019} we may then extend $\psi$ to a Borel embedding of the whole system $(Y,S)$ satisfying $\chi\circ \psi=\Id_Y$.
\end{rem}

 We finish now  the proof of the second item of Theorem \ref{singular}. The only recurrent points in $\overline{\psi(Z\cap Y)}\setminus \psi(Z\cap Y)$ belong to the two subshifts of finite type generated \footnote{ In other terms any element of the subshift of finite type is a concatenation  of these two words.} respectively by the $2$-words $2^02^\times$, $2^12^\times$ and  $4^04^{\times}$, $4^14^{\times}$, which  are contained in $\chi^{-1}*$. Therefore $h^\chi(\mu)=h(\mu)$ for any $\mu\in \mathcal M_{S}(Y)$ with $\mu(Y\cap Z)=1$. 
 Moreover $h^\chi(\delta_*)=h_{top}(\chi^{-1}*)=\frac{\log 2}{2}$. The function $h^\chi$ being affine we get $h^\chi(\mu)=h(\mu)+\mu(*)\frac{\log 2}{2}$ for all $\mu \in \mathcal M_{S}(Y)$.\\
 
 
\underline{{\it Step 3 :} Conclusion.} 
  Finally the  symbolic extension  $\pi'$ induced by $\chi$  of $(Y^{f'}, S^{f'})$  by the suspension flow over 
 $(\overline{\psi(Y\cap Z)}, \sigma)$ under the roof function $f'\circ \chi$ satisfies  $h^{\pi'}(\nu)=h(\nu)$ for all $\nu\in \mathcal M_{S^{f'}}^*(Y^{f'})$ and $h^{\pi'}(\delta_{*^{f'}})=\frac{h^{\chi}(\delta_*)}{f'(*)}=\frac{1}{2l}$. As the extension $\pi:(Y^{f'}, S^{f'})\rightarrow (X^f, T^f)$ is  principal and satisfies $\pi^{-1}\mathcal M_{X^{f}}^*(X^{f})=\mathcal M_{S^{f'}}^*(Y^{f'})$, the extension $\Pi=\pi\circ \pi'$ also satisfies $h^{\Pi}(\nu)=h(\nu)$ for all $\nu\in \mathcal M_{T^{f}}^*(X^{f})$ and $h^{\Pi}(\delta_{*^{f}})=\frac{1}{2l}$. This concludes the proof of the second item of Theorem \ref{singular}, because the function $h^{\Pi}$ is affine.
\end{proof}

\bibliographystyle{alpha}
\bibliography{universal_bib}

\end{document}